\documentclass{article}%
\usepackage{amsfonts}
\usepackage{amsmath}
\usepackage{amssymb}
\usepackage{graphicx}%
\numberwithin{equation}{section}
\newtheorem{theorem}{Theorem}[section]

\newtheorem{lemma}[theorem]{Lemma}

\newtheorem{remark}[theorem]{Remark}

\newcommand{\sign}{ \ensuremath{\mathrm{sgn} }}
\newcommand{\rank}{ \ensuremath{\mathrm{rank} }}

\newcommand{\spann}{ \ensuremath{\mathrm{span} }}

\newenvironment{proof}[1][Proof]{\noindent\textbf{#1.} }{\ \rule{0.5em}{0.5em}}
\begin{document}

\title{Asymptotics for eigenvalues of a non-linear integral system
}
\author{D.E.Edmunds
\and J.Lang} \maketitle

\begin{abstract}
Let $I=[a,b]\subset\mathbb{R},$ let $1< q , p<\infty,$ let $u$ and
$v$ be positive functions with $u\in L_{p^{\prime}}(I),$ $v\in
L_{q}(I)$ and let $T:L_p(I) \to L_q(I)$ be the Hardy-type operator
given by
\[
(Tf)(x)=v(x)\int\nolimits_{a}^{x}f(t)u(t)dt,\text{ }x\in I.
\]
We show that the asymptotic behavior of the eigenvalues $\lambda$
of the non-linear integral system
\[ g(x)=(Tf)(x) \qquad (f(x))_{(p)}=\lambda(T^{\ast}(g_{(p)}))(x)
\]
( where, for example, $t_{(p)}=|t|^{p-1} \sign(t)$) is given by
\[
\lim_{n\rightarrow\infty} n \hat{\lambda}_n(T)=c_{p,q}
\left(  \int\nolimits_{I}(uv)^{r}%
dt\right)  ^{1/r},\text{ for } 1<q<p<\infty.
\]
\[
\lim_{n\rightarrow\infty} n \check{\lambda}_n(T)=c_{p,q}
\left(  \int\nolimits_{I}(uv)^{r}%
dt\right)  ^{1/r},\text{ for } \ 1<p<q<\infty,
\]
Here $r=1/p^{\prime}+1/p$, $c_{p,q}$ is an explicit constant
depending only on $p$ and $q$, $\hat{\lambda}_n(T) = \max
(sp_n(T,p,q))$, $\check{\lambda}_n(T) = \min (sp_n(T,p,q))$ where
$sp_n(T,p,q)$ stands for the set of all eigenvalues $\lambda$
corresponding to eigenfunctions $g$ with $n$ zeros.
\end{abstract}

\section{Introduction and preliminaries}
Through this paper we shall assume $I=[a,b],$ where $-\infty<a<b<\infty,$ and let
$p,q \in (1,\infty)$, $(x)_{(p)}%
:=\left\vert x\right\vert ^{p-1}$ sgn $(x),$ $x\in\mathbb{R}$ and
$1/p'=1-1/p.$

Let $u$ and $v$ be positive functions on $I$, with $u\in
L_{p^{\prime}}(I),$ $v\in L_{q}(I)$.

Define the Hardy-type operator $T:L_p(I) \to L_q(I)$ by
\[
(Tf)(x)=v(x)\int\nolimits_{a}^{x}f(t)u(t)dt,\text{ }x\in I.
\]
Such maps have been intensively studied: see \cite[Chapter
2]{EE1}.

Since $|I|=b-a <\infty,$ $u\in L_{p'}(I)$ and $v\in L_q(I)$ then
$T$ is compact, see \cite[chapter 2]{EE2}.

As more detailed information about the native of the compactness
of a map is provided by its approximation, Kolmogorov and
Bernstein numbers, much attention has been paid to the asymptotic
behavior of these numbers for the map $T$. The analysis is
decidedly easier when $p=q$, and an account of the situation in
this case is given in \cite{EE2}. For the case $p \not = q$ we
refer to \cite{EL1}, \cite{EL2} in which a key role is played by
the non-linear integral system:
\begin{equation}
g(x)=(Tf)(x) \label{Eq 2.2}%
\end{equation}
and
\begin{equation}
(f(x))_{(p)}=\lambda(T^{\ast}(g_{(q)}))(x), \label{Eq 2.3}%
\end{equation}
where $g_{(q)}$ is the function with value $(g(x))_{(q)}$ at $x$
and $T^{\ast}$ is the map defined by
$(T^{\ast}f)(x)=u(x)\int\nolimits_{x}^{b}v(y)f(y)dy.$

The  non-linear system (\ref{Eq 2.2}) and (\ref{Eq 2.3}) gives us
the following non-linear equation:
\begin{equation}
(f(x))_{(p)}=\lambda T^{\ast}((Tf)_{(q)})(x). \label{Eq 2.3-one}%
\end{equation}

This is equivalent to its dual equation:
\begin{equation}
(s(x))_{(q')}=\lambda^{\ast}T((T^{\ast}s)_{(p')})(x) \label{Eq 2.3-dual}.%
\end{equation}
And we have this relation: For given $f$ and $\lambda$ satisfying
(\ref{Eq 2.3-one}) we have $s=(Tf)_{(q)}$ and $\lambda^{\ast}=
\lambda_{(p')}$ satisfying (\ref{Eq 2.3-dual}), and for given $s$
and $\lambda^*$ satisfying (\ref{Eq 2.3-dual})we have
$f=(T^*s)_{(p')}$ and $\lambda= \lambda^{\ast}_{(q)}$ satisfying
(\ref{Eq 2.3-one}).

 By a
spectral triple will be meant a triple $(g, f,\lambda)$ satisfying
(\ref{Eq 2.2}) and (\ref{Eq 2.3}), where $\left\Vert f\right\Vert
_{p}=1$; $(g, \lambda)$ will be called a spectral pair; the
function $g$ corresponding to $\lambda$ is called a spectral
function and the number $\lambda$ occurring in a spectral pair
will be called a spectral number.

For the system (\ref{Eq 2.2}) and (\ref{Eq 2.3}) we denote by
$SP(T,p,q)$ the set of all spectral triples; $sp(T,p,q)$ will
stand for the set of all spectral numbers $\lambda$ from
$SP(T,p,q)$.

It can be seen that this non-linear system is related to the
isoperimetric problem of determining

\begin{equation}
\sup_{g\in T(B)}\left\Vert g\right\Vert _{q}, \label{Eq 2.1}%
\end{equation}

where $B:=\{f\in L_{p}(I):\left\Vert f\right\Vert _{p}\leq1\}.$

Moreover, this problem can be seen as a natural generalization of
the $p,q-$Laplacian differential equation. For if $u$ and $v$ are
identically equal to $1$ on $I$, then (\ref{Eq 2.2}) and (\ref{Eq
2.3}) can be transformed into the $p,q-$Laplacian differential
equation:
\begin{equation}
-\left(  \left(  w^{\prime}\right)  _{(p)}\right)
^{\prime}=\lambda(w)_{(q)},
\label{Eq 2.4}%
\end{equation}
with the boundary condition
\begin{equation}
w(a)=0. \label{Eq 2.5}%
\end{equation}

  If $ g, f$ and $\lambda$ satisfy (\ref{Eq
2.2}) and (\ref{Eq 2.3}) then, the integrals being over $I$,
\begin{align*}
\int |g(x)|^q dx & = \int g (g)_{(q)} dx = \int Tf(x) (g)_{(q)} dx \\
& = \int f(x) T^*(g)_{(q)} dx=  \lambda^{-1} \int f(x) (f)_{(p)} \\
& = \lambda^{-1} \int |f(x)|^p dx  .
 \end{align*}
From this it follows that $\lambda^{-1}= \|g\|_q^q / \|f\|_p^p$
and then for $(g_1,\lambda_1) \in SP(T,p,q)$ we have
$\lambda_1^{-1/q}=\|g_1\|_q$.

 Given any continuous function $f$ on $I$ we denote by $Z(f)$ the
number of distinct zeros of $f$ on $\overset{o}{I,}$ and by $P(f)$
the number of sign changes of $f$ on this interval. The set of all
spectral triples $(g,f,\lambda)$ with $Z(g)=n$
$(n\in\mathbb{N}_{0})$ will be denoted by $SP_{n}(T,p,q),$ and
$sp_{n}(T,p,q)$ will represent the set of all corresponding
numbers $\lambda.$ We set $\hat{\lambda}_n = \max sp_{n}(T,p,q)$
and $\check{\lambda}_n = \min sp_{n}(T,p,q).$

Our main result is that the asymptotic behavior of the
$\hat{\lambda}_n$ can be determined when $1<q<p<\infty$: we show
that
\[
\lim_{n\rightarrow\infty} n \hat{\lambda}_n(T)=c_{p,q}
\left(  \int\nolimits_{I}(uv)^{r}%
dt\right)  ^{1/r},
\]
where $r=1/p'+1/q$ and $c_{pq}$ is a constant whose dependence on
$p$ and $q$ is given explicitly. A corresponding result holds for
$\check{\lambda}_n$ when $1<p<q<\infty$. Moreover, $sp_n(T,p,p)$
contains exactly one element, so that in this case
$\hat{\lambda}_n=\check{\lambda}_n= {\lambda}_n$ say, and the
asymptotic behavior of the ${\lambda}_n$ is given by the formula
above.

We now give some results to prepare for the major theorems in \S 2
and \S 3.

\begin{lemma}
\label{lemma 3.6 BT} Let $f\neq 0$ be a function on $[a,b]$ such
that $Tf(a)=Tf(b)=0$. Then $P(f) \ge 1$.
\end{lemma}
\begin{proof}
This follows from the positivity of $T$ and Rolle's theorem.
\end{proof}

\begin{lemma}
\label{corollary 3 BT} Let $(g_i,f_i,\lambda_i) \in SP(T,p,q),$
$i=1,2$, $1<p,q <\infty$. Then for any $\varepsilon >0$,
\begin{equation}
P(Tf_1-\varepsilon Tf_2) \le P(Tf_1 -
\varepsilon^{(p-1)/(q-1)}(\lambda_2/\lambda_1)^{1/(q-1)} Tf_2).
\label{3 BT}
\end{equation}
If the function $f_1-\varepsilon f_2$ has a multiple zero and $
P(Tf_1 - \varepsilon^{(p-1)/(q-1)}(\lambda_2/\lambda_1)^{q/(q-1)}
Tf_2) <\infty$, then the inequality (\ref{3 BT}) is strict.
\end{lemma}
\begin{proof} We will use Lemma \ref{lemma 3.6 BT} and the fact
that $\sign(a-b)=\sign((a)_{(p)}-(b)_{(p)})$.
\begin{align*}
 P(Tf_1-\varepsilon Tf_2) & \le Z(Tf_1-\varepsilon Tf_2) \le P(f_1-\varepsilon
 f_2) \\
 & \le P((f_1)_{(p)}-\varepsilon^{p-1}(f_2)_{(p)}) \\
 & (\text{ use } (\ref{Eq 2.3-one}) \text{ for } f_1 \text{ and } f_2), \\
 & \le P(\lambda_1
 T^*((g_1)_{(q)}) -\varepsilon^{p-1}  \lambda_2
 T^*((g_2)_{(q)})) \\
& \le Z(\lambda_1
 T^*((g_1)_{(q)}) -\varepsilon^{p-1}  \lambda_2
 T^*((g_2)_{(q)})) \\
& \le P(
 (g_1)_{(q)} -\varepsilon^{p-1} ( \lambda_2 / \lambda_1)
 (g_2)_{(q)}) \\
 & \le P( g_1 -\varepsilon^{(p-1)/(g-1)} (\lambda_2 /
 \lambda_1)^{1/(q-1)}
 g_2) \\
 & \le P(Tf_1 -
\varepsilon^{(p-1)/(q-1)}(\lambda_2/\lambda_1)^{1/(q-1)} Tf_2).
\end{align*}

\end{proof}

\begin{theorem}
\label{Theorem 2.2}For all $n\in\mathbb{N}$,
$SP_{n}(T,p,q)\neq\emptyset.$
\end{theorem}
\begin{proof}
This essentially follows ideas from \cite{BT1} (see also
\cite{N1}), but we give the details for the convenience of the
reader. For simplicity we suppose that $I$ is the interval
$[0,1].$ A key idea in the proof is the introduction of an
iterative procedure used in \cite{BT1}.

Let $n\in\mathbb{N}$ and define%
\[
\mathcal{O}_{n}=\left\{  z=(z_{1},...,z_{n+1})\in\mathbb{R}^{n+1}%
:\sum\nolimits_{i=1}^{n+1}\left\vert z_{i}\right\vert =1\right\}
\]
and%
\[
f_{0}(x,z)=sgn(z_{j})\text{ for
}\sum\nolimits_{i=0}^{j-1}\left\vert z_{i}\right\vert
<x<\sum\nolimits_{i=1}^{j}\left\vert z_{i}\right\vert ,\text{
}j=1,...,n+1,\text{ with }z_{0}=0.
\]
With $g_{0}(x,z)=Tf_{0}(x,z)$ we construct the iterative process%
\[
g_{k}(x,z)=Tf_{k}(x,z),\text{ }f_{k+1}(x,z)=(\lambda_{k}(z)T^{\ast}%
(g_{k}(x,z))_{(q)})_{(p^{\prime})},
\]
where $\lambda_{k}$ is a constant so chosen that%
\[
\left\Vert f_{k+1}\right\Vert _{p}=1
\]
and $1/p+1/p^{\prime}=1.$ Then, all integrals being over $I,$%
\begin{align*}
1 &  =\int\left\vert f_{k}(x,z)\right\vert ^{p}dx=\int f_{k}(f_{k}%
)_{(p)}dx=\int f_{k}\left(
[\lambda_{k-1}T^{\ast}((g_{k-1})_{(q)}\right)
]_{(p^{\prime})})_{(p)}dx\\
&  =\int f_{k}\lambda_{k-1}T^{\ast}((g_{k-1})_{(q)})dx\\
&  =\lambda_{k-1}\int T(f_{k})(g_{k-1})_{(q)}dx\leq\lambda_{k-1}%
\left\Vert g_{k}\right\Vert _{q}\left\Vert g_{k-1}\right\Vert _{q}^{q-1}%
\end{align*}
and also%
\begin{align*}
\left\Vert g_{k-1}\right\Vert _{q}^{q} &  =\int\left\vert g_{k-1}%
(x,z)\right\vert ^{q}dx=\int(g_{k-1})_{(q)}g_{k-1}dx\\
&  =\int(g_{k-1})_{(q)}T(f_{k-1})dx=\int T^{\ast}((g_{k-1})_{(q)})f_{k-1}dx\\
&  =\lambda_{k-1}^{-1}\int\lambda_{k-1}T^{\ast}((g_{k-1})_{(q)}%
)f_{k-1}dx\\
&  \leq\lambda_{k-1}^{-1}\left(  \int\left\vert (\lambda_{k-1}
T^{\ast }((g_{k-1})_{(q)})_{(p^{\prime})}\right\vert
^{p^{\prime}}dx\right) ^{1/p^{\prime}}\left(  \int\left\vert
f_{k-1}\right\vert ^{p}dx\right)
^{1/p}\\
&  =\lambda_{k-1}^{-1}\left(  \int\left\vert (\lambda_{k-1}T^{\ast
}((g_{k-1})_{(q)})_{(p^{\prime})}\right\vert
^{p^{\prime}}dx\right)
^{1/p^{\prime}}\\
&  =\lambda_{k-1}^{-1}\left(  \int\left\vert f_{k}\right\vert
^{p}dx\right) ^{1/p}=\lambda_{k-1}^{-1}.
\end{align*}
From these inequalities it follows that%
\[
\left\Vert g_{k-1}(\cdot,z)\right\Vert _{q}\leq\lambda_{k-1}^{-1/q}%
\leq\left\Vert g_{k}(\cdot,z)\right\Vert _{q}.
\]
This shows that the sequences $\{g_{k}(\cdot,z)\}$ and
$\{\lambda_{k}^{-1/q}(z)\}$ are monotonic increasing. Put
$\lambda(z)=\lim_{k\rightarrow\infty}\lambda _{k}(z);$ then
$\left\Vert g_{k}(\cdot,z)\right\Vert _{q}\rightarrow
\lambda^{-1/q}(z).$

As the sequence $\{f_{k}(\cdot,z)\}$ is bounded in $L_{p}(I),$
there is a subsequence $\{f_{k_{i}}(\cdot,z)\}$ that is weakly
convergent, to $f(\cdot,z),$ say. Since $T$ is compact,
$g_{k_{i}}(\cdot,z)\rightarrow Tf(\cdot,z):=g(\cdot,z)$ and we
also have $f(\cdot,z)=\left(  \lambda
(z)T^{\ast}(g(\cdot,z)\right)  _{(q)})_{(p^{\prime})}.$ It follows
that for each $z\in\mathcal{O}_{n},$ the sequence
$\{g_{k_{i}}(\cdot,z)\}$ converges to a spectral function.

Now set $z=(0,0,...,0,1)\in\mathcal{O}_{n}.$ Then
$f_{0}(\cdot,z)=1,$ and as the operators $T$ and $T^{\ast}$ are
positive, $g_{k}(\cdot,z)\geq0$ for all $k,$ so that
$g(\cdot,z)\geq0.$ Thus $(g(\cdot,z), f(\cdot,z), \lambda(z))\in
SP_{0}(T,p,q):$ $SP_{0}(T,p,q)\neq\emptyset.$

Next we show that for all
$n\in\mathbb{N},SP_{n}(T,p,q)\neq\emptyset.$ Given
$n,k\in\mathbb{N},$ set%
\[
E_{k}^{n}=\{z\in\mathcal{O}_{n}:Z(g_{k}(\cdot,z))\leq n-1\}.
\]
From the definition of $T$ it follows that $g_{k}(\cdot,z)$
depends continuously on $z;$ thus $E_{k}^{n}$ is an open subset of
$\mathcal{O}_{n}$ and $F_{k}^{n}:=\mathcal{O}_{n}\backslash
E_{k}^{n}$ is a closed subset of
$\mathcal{O}_{n}.$ Let $0<t_{1}<...<t_{n}<1$ and put%
\[
F_{k}(\alpha)=(g_{k}(t_{1},\alpha),...,g_{k}(t_{n},\alpha)),\text{
}\alpha \in\mathcal{O}_{n}.
\]
Then $F_{k}$ is a continuous, odd mapping from $\mathcal{O}_{n}$
to
$\mathbb{R}^{n}.$ By Borsuk's theorem, there is a point $\alpha_{k}%
\in\mathcal{O}_{n}$ such that $F_{k}(\alpha_{k})=0;$ that is,
$\alpha_{k}\in F_{k}^{n}.$ From the definition of $g_{k}$ and
$f_{k+1},$ together with the
positivity of $T$ and $T^{\ast},$ we have%
\[
Z(g_{k+1})\leq P(f_{k+1})\leq Z(f_{k+1})\leq P(g_{k})\leq
Z(g_{k}),
\]
so that $E_{k}^{n}\subset E_{k+1}^{n},$ which implies that
$F_{k}^{n}\supset
F_{k+1}^{n}.$ Hence there exists $\widetilde{\alpha}\in\cap_{k\geq1}F_{k}%
^{n},$ and as above we see that $g_{k}(\cdot,\widetilde{\alpha})$
converges, as $k\rightarrow\infty,$ to a spectral function
$g(\cdot,\widetilde{\alpha })\in SP_{n}(T,p,q).$ Thus
$SP_{n}(T,p,q)\neq\emptyset$ and the proof is complete.
\end{proof}

We note that the previous theorem is true for much more general
integral operators (i.e. integral operators with totally positive
kernel, see \cite{N1}).

 We now define Kolmogorov widths $d_n(T)$ for $T$ as a map from
$L_{p}(I)$ to $L_{q}(I)$ when $1<q , p<\infty.$ These numbers are
defined by:
\[
d_n(T)=d_{n}=\inf_{X_{n}}\sup_{\|f\|_{p,I}\le 1} \inf_{g\in X_n}
\left\Vert Tf-g \right\Vert _{q,I}/\left\Vert f\right\Vert _{p,I},
\qquad n\in \mathbb{N}
\]
where the infimum is taken over all $n$-dimensional subspaces
$X_{n}$ of $L_{q}(I)$.

To get an upper estimate for eigenvalues via the Kolmogorov
numbers, we start by recalling the Makovoz lemma (see 3.11 in
\cite{BT1}).

\begin{lemma}
\label{Makovoz}Let $U_n \subset \{Tf;\|f\|_{p,I} \le 1 \}$ be a
continuous and odd image of the sphere $S^n$ in $\mathbb{R}^n$
endowed with the $l_1$ norm. Then
\[ d_n(T) \ge \inf \{ \|x\|_{q,I}, x\in U_n\}
\]
\end{lemma}

\begin{lemma}
  \label{dn lambdan} If $n >1$, then
  $ d_n(T) \ge \hat{\lambda}^{-1/q} $ where
  $\hat{\lambda}= \max \{\lambda \in \cup_{i=0}^n
sp_i(p,q)\}.$
\end{lemma}

\begin{proof}Let us denote
$\hat{\lambda}= \max \{\lambda \in \cup_{i=0}^n sp_i(p,q)\}.$ The
iteration process from the proof of Theorem \ref{Theorem 2.2}
gives us for each $k\in \mathbb{N}$ and $z\in \mathcal{O}_{n}$ a
function $g_k(.,z)$. By the Makavoz lemma we have
\begin{equation} d_n(T) \ge \max_{k\in \mathbb{N}} \min_{z\in \mathcal{O}_{n}}
\|g_k(.,z)\|_{q,I}. \label{a}
\end{equation}
Let us suppose that we have
\begin{equation}  \min_{z\in \mathcal{O}_{n}}
\lim_{k\to \infty} \|g_k(.,z)\|_{q,} = \max_{k\in \mathbb{N}}
\min_{z\in \mathcal{O}_{n}} \|g_k(.,z)\|_q. \label{b}
\end{equation}
 Then from (\ref{a}) and (\ref{b}) it follows that
 \[ d_n(T) \ge \min_{z\in \mathcal{O}_{n}}
\lim_{k\to \infty} \|g_k(.,z)\|_{q} \ge \hat{\lambda}^{-1/q},
\]
since $ \lim_{k\to \infty} g_k(.,z) \in SP(T,p,q).$ We have to
prove (\ref{b}). From the monotonicity of $\|g_k(.,z)\|_{q,I}$ we
have
\[\max_{k\in \mathbb{N}} \min_{z\in
\mathcal{O}_{n}} \|g_k(.,z)\|_q = \lim_{k\to \infty} \min_{z\in
\mathcal{O}_{n}}
 \|g_k(.,z)\|_{q,}. \]

From $\max \ \min \le \min \ \max$ it follows that
\[l:=\lim_{k\to \infty} \min_{z\in
\mathcal{O}_{n}}
 \|g_k(.,z)\|_{q,}=\max_{k\in \mathbb{N}} \min_{z\in
\mathcal{O}_{n}} \|g_k(.,z)\|_q \]
\[ \le \min_{z\in \mathcal{O}_{n}}
\max_{k\in \mathbb{N}}  \|g_k(.,z)\|_q=  \min_{z\in
\mathcal{O}_{n}} \lim_{k\to \infty}
 \|g_k(.,z)\|_{q,}=:h
 \]
Denote $H_k(\varepsilon)=\{z\in \mathcal{O}_{n}; \|g_k(.,z)\|_q\le
h-\varepsilon \}$ where $0<\varepsilon \le h$.

Since the mapping $z \mapsto g_k(.,z)$ is continuous,
$H_k(\varepsilon)$ is a closed subset of $\mathcal{O}_{n}$, and
from the construction of the sequence $g_k$ we see that
$H_0(\varepsilon) \supset H_1(\varepsilon) \supset ...$.

If $y_0\in \cap_{k\in \mathbb{N}} H_k(\varepsilon) \neq \emptyset$
then $h= \min_{z\in \mathcal{O}_{n}} \lim_{k\to \infty}
 \|g_k(.,z)\|_{q} \le \lim_{k\to \infty} \|g_k(.,y_0)\|_q\le
 h-\varepsilon $ is a contradiction. Then there exist $k_0 \in \mathbb{N}$
 such that $H_k(\varepsilon) = \emptyset$ for $k \ge k_0$ and
$\min_{z\in \mathcal{O}_{n}}  \|g_k(.,z)\|_{q} \ge h-\varepsilon$
for $k \ge k_0$. Then we have that $h=l$ and (\ref{b}) is proved.
\end{proof}

Next we define Bernstein widths which will help us in section 3.
The Bernstein widths $b_n(T)$ for $T$: $L_p(I) \to L_q(I)$ when
$1<p,q<\infty$ are defined by:
\[b_n(T):=\sup_{X_{n+1}} \inf_{Tf \in X_{n+1} \setminus \{0\}}
\|Tf\|_{q,I} / \|f\|_{p,I}, \] where the supremum is taken over
all subspaces $X_{n+1}$ of $T(L_p(I))$ with dimension $n+1$. Since
$u$ and $v$ are positive functions, the Bernstein widths can be
expressed as
\[b_n(T)=\sup_{X_{n+1}} \inf_{\alpha \in \mathbb{R}^n \setminus \{0\}}
{{\|T \left( \sum_{i=1}^{n+1} \alpha_i f_i \right) \|_{q,I}} \over
{\| \sum_{i=1}^{n+1} \alpha_i f_i \|_{p,I}}},
\]
where the supremum is taken over all $(n+1)$-dimensional subspaces
$X_{n+1}= \spann \{f_1, ... ,f_{n+1}\} \subset L_p(I)$.

Now we use techniques from Theorem \ref{Theorem 2.2} to obtain an
upper estimate for the Bernstein widths.

\begin{lemma}
  \label{bn<lambdan} If $n >1$ then
  $ b_n(T) \le \check{\lambda}^{-1/q} $, where
  $\check{\lambda} = \min(
sp_n(p,q)).$
\end{lemma}

\begin{proof}
Suppose there exists a linearly independent system of functions
$\{f_1, ... , f_{n+1}\}$ on $I$,  such that:
\[  \min_{\alpha \in \mathbb{R}^n \setminus \{0\}}
{{\|T \left( \sum_{i=1}^{n+1} \alpha_i f_i \right) \|_{q,I}} \over
{\| \sum_{i=1}^{n+1} \alpha_i f_i \|_{p,I}}} >
\check{\lambda}^{-1/q}.
\]

Let us define the $n-$dimensional sphere
\[ O_n=\left\{ T \left( \sum_{i=1}^{n+1} \alpha_i f_i \right),
\| \sum_{i=1}^{n+1} \alpha_i f_i \|_{p,I}=1 \right\}.
\]

Let $g_0(.)\in O_n$ and define a sequence of functions $h_k(.),
g_k(.)=g_k(.,g_0), k\in \mathbb{N}$, according to the following
rule:
\[ g_k(x)=Th_k(x), \qquad h_{k+1}(x)=(\lambda_k
T^{*}(g_k(x))_{(q)})_{(p')}, \]
 where $\lambda_k>0$ is a constant
chosen so that $\|h_{k+1}\|_{p,I}=1$.

We denote $O_n(k)=\{h_k(.,h_0), h_0(.) \in O_n\}$. As in the proof
of Theorem \ref{Theorem 2.2} we have:

$\|g_k\|_{q,I}$ is a nondecreasing as $k\nearrow \infty$. For each
$k \in \mathbb{N}$ there exists $g_k \in O_n(k)$ with $n$ zeros
inside $I$; $\lim_{k \to \infty} g_k(.,g_0)$ is an eigenfunction
and there exists $g_0(.)$ such that $\lim_{k \to \infty}
g_k(.,g_0)$ is an eigenfunction with $n$ zeros. Moreover
$\lambda_k $ is monotonically decreasing as $k\nearrow \infty$.

Let $\overline{\alpha}\in \mathbb{R}^{n+1}$ be such that:
$\overline{g_0}(.)= \left( \sum_{i=1}^{n+1} \overline{\alpha_i}
f_i \right)$ is a function for which $\lim_{k \to \infty}
\overline{g_k}(.,g_0)$ is an eigenfunction with $n$ zeros.

Then we have the following contradiction:
\begin{align*}
\min_{\alpha \in \mathbb{R}^n \setminus \{0\}} {{\|T \left(
\sum_{i=1}^{n+1} \overline{\alpha_i} f_i \right) \|_{q,I}} \over
{\| \sum_{i=1}^{n+1} \overline{\alpha_i} f_i \|_{p,I}}} & \le
\|\overline{g_0}(.)
\|_{q,I} \\
\le \lim_{k \to \infty} \| {g_k}(.,\overline{g_0}(.)) \|_{q,I} &
\le \check{\lambda}^{-1/q},
\end{align*}

\end{proof}

In the next two sections we obtain an upper estimate for
Kolmogorov numbers and a lower estimate for Bernstein numbers. We
shall need the approximation numbers $a_n(T)$ of $T$, defined by
$a_n(T)=\inf \|T-F\|$, where the infimum is taken over all linear
operators $F$ with rank at most $n-1$.

\section{The case $q\le p$}

We recall Jensen's inequality (see, for example \cite{KJF}, p.133)
which will be of help in the next lemma.

\begin{theorem}
\label{Jenssen}If $F$ is a convex function, and $h(.) \ge 0$ is a
function such that $\int_I h(t) dt=1$, then for every non-negative
function $g$,
$$ F(\int_I h(t) g(t) dt) \le \int_I h(t) F(g(t)) dt. $$
\end{theorem}

The following lemma give us a lower estimate for eigenvalues.

\begin{lemma}
  \label{an lambdan} If $n >1$ then
  $ a_n(T) \le \widehat{\lambda}^{-1/q} $, where
  $\widehat{\lambda}=  \max(sp_n(p,q)).$
\end{lemma}
\begin{proof}
For the sake of simplicity we suppose that $|I|=1.$

Let $(\widehat{g},\widehat{f},\widehat{\lambda}) \in SP_n(T,p,q)$.
Denote by $\{a_i\}_{i=0}^n$ the set of zeros of $\widehat{g}$
(with $a_0=a$) and by $\{b_i\}_{i=1}^{n+1}$ (with $b_{n+1}=b$) the
set of zeros of $\widehat{f}$. Set $I_i=(b_i,b_{i+1})$ for
$i=1,...,n$ and $I_0=(a_0,b_1)$, and define
$$T_nf(x):= \sum_{i=0}^n \chi_{I_i}(.) v(.) \int_a^{a_i}
u(t) f(t) dt. $$
 Then the $\rank $ of $T_n$ is at most  $n$.

We have (see \cite[Chapter 2]{EE1}) $d_n(T)\le a_n(T) \le
\sup_{\|f\|_p \le 1} \|Tf-T_nf\|_q.$

Let us consider the extremal problem:
\begin{equation}
\sup_{\|f\|_p \le 1} \|Tf-T_nf\|_q.
\label{extrem1}%
\end{equation}
We can see that this problem is equivalent to
\begin{equation}
\sup \{ \|Tf\|_q: \|f\|_p \le 1, (Tf)(a_i)=0 \mbox{ for } i=0
\dots n \}
\label{extrem2}%
\end{equation}

 Since $T$ and $T_n$ are compact then there is a
solution of this problem, that is, the supremum is attained. Let
$\bar{f}$ be one such solution and denote $\bar{g}=T \bar{f}$. We
can choose $\bar{f}$ such that $\bar{g}(t) \widehat{g}(t) \ge 0,$
for all $t\in I$. We have $\| \bar{g} \|_{q,I} \ge  \|
\widehat{g}\|_{q,I}$

Note that for any $f \in L^p(I)$ such that $Tf(a_i)=0$ for every
$i=0,...,n$ we have $Tf(x)=T^+f(x)$ for each $x\in I$, where
\[T^+f(x):=\int_I K(x,t) f(t) dt = \sum_{i=0}^n \chi_{I_i}(.) v(.) \int_a^{x}
u(t) f(t) dt
\]
and
\[  K(x,t):=\sum_{i=0}^n \chi_{I_i}(x) v(x)u(t)
\chi_{(a_i,x)} \sign(x-a_i).\]

Set $s(t)=|\hat{g}(t)|^q \hat{\lambda}^q,$ where $\hat{\lambda}=
\|\hat{g}\|_{q,I}$. Then, all integrals being over $I$, we have

\begin{align*}
 \left( \int  |\bar{g}(t)|^{q}  dt \right)^{1/q}  = &
\hat{\lambda}^{-1/q}\left( \int s(t) \left|{\bar{g}(t) \over
\hat{g}(t)}\right|^{q} dt \right)^{1/q}
\\
& \qquad \mbox{(use Jensen's inequality, noting that} \int s(t)
dt=1
 ) \\
 \le &
\hat{\lambda}^{-1/q}\left( \int s(t) \left|{\bar{g}(t)
\over \hat{g}(t)}\right|^{p} dt \right)^{1/p}\\
= & \hat{\lambda}^{-1/q}\left( \int s(t) \left|{T^+\bar{f}(t)
\over \hat{g}(t)}\right|^{p} dt \right)^{1/p}\\
= & \hat{\lambda}^{-1/q}\left( \int s(t) \left|{\int
K(t,\tau)\bar{f}(\tau) d\tau \over \hat{g}(t)}\right|^{p} dt
\right)^{1/p} \\
 = & \hat{\lambda}^{-1/q}\left(
\int s(t) \left|\int {K(t,\tau)\hat{f}(\tau) \over \hat{g}(t)}
{\bar{f}(\tau) \over \hat{f}(\tau)} d\tau \right|^{p}
dt \right)^{1/p} \\
& \qquad \mbox{ (use Jensen's inequality, noting that}\\
& \left. {K(t,\tau)\hat{f}(\tau) \over \hat{g}(t)} \ge 0 \mbox{
and }
\int {K(t,\tau)\hat{f}(\tau) \over \hat{g}(t)}d\tau =1 \right)\\
 \le & \hat{\lambda}^{-1/q}\left( \int s(t) \int
{K(t,\tau)\hat{f}(\tau) \over \hat{g}(t)} \left| {\bar{f}(\tau)
\over \hat{f}(\tau)}\right|^{p} d\tau dt \right)^{1/p} \\
\end{align*}

\begin{align*}
= & \hat{\lambda}^{-1/q}\left( \int \left| {\bar{f}(\tau) \over
\hat{f}(\tau)}\right|^{p}  \bar{f}(\tau) \int
{K(t,\tau)s(t) \over \hat{g}(t)} dt d\tau \right)^{1/p} \\
= & \hat{\lambda}^{-1/q}\left( \int \left| {\bar{f}(\tau) \over
\hat{f}(\tau)}\right|^{p}  \bar{f}(\tau) \int
{K(t,\tau)|\hat{g}(t)|^q  \over \hat{g}(t)}\hat{\lambda} dt
d\tau \right)^{1/p} \\
= & \hat{\lambda}^{-1/q}\left( \int \left| {\bar{f}(\tau) \over
\hat{f}(\tau)}\right|^{p}  \bar{f}(\tau) \int {K(t,\tau)
\hat{g}_{(q)}(t) } dt \hat{\lambda} d\tau
\right)^{1/p} \\
& \qquad \left( \mbox{ use } \int {K(t,\tau) \hat{g}_{(q)}(t) } dt
\hat{\lambda}^q=  \hat{\lambda}
T^*(\hat{g}_{(q)})(t)=\hat{f}_{(p)}(t) \right) \\
= & \hat{\lambda}^{-1/q}\left( \int \left| {\bar{f}(\tau) \over
\hat{f}(\tau)}\right|^{p}  \hat{f}(\tau)
\hat{f}_{(p)}(\tau) d\tau\right)^{1/p}\\
& \qquad ( \mbox{ use }  \hat{f}(t)
\hat{f}_{(p)}(t) = | \hat{f}(t)|^p ) \\
 = & \hat{\lambda}^{-1/q}\left( \int \left|
\bar{f}(\tau) \right|^{p} d\tau \right)^{1/p}=
\hat{\lambda}^{-1/q}.
\end{align*}

From this it follows that $a_n(T) \le \hat{\lambda}^{-1/q}$.
\end{proof}

\begin{theorem}
  \label{q<p} If $1< q\le p < \infty$, then
  \[ \lim_{n \to \infty} n {\hat{\lambda}_n}^{-1/q}
  =c_{pq}\left(\int_I |uv|^{1/r} dt \right)^r\]
  where $r=1/p'+1/q$, ${\hat{\lambda}_n}=
   \max(sp_n(p,q))$ and
\begin{equation}
 c_{pq}= {(p')^{1/q} q^{1/p'} (p'+q)^{1/p-1/q} \over 2
   B(1/q,1/p')} \label{constant}
\end{equation}
   ($B$  denotes the Beta
   function).
\end{theorem}
\begin{proof}
From \cite{EL2} we have
\[\lim_{n \to \infty} n
{a_n(T)}= \lim_{n \to \infty} n {d_n(T)}
  =c_{pq}\left(\int_I |uv|^{1/r} dt \right)^r \] and since $d_n(T) \le
  a_n(T)$, $a_n(T)\searrow 0$ and $d_n(T)\searrow 0$
then from Lemma \ref{an lambdan} follows:
\[c_{pq}\left(\int_I |uv|^{1/r} dt
\right)^r \le \liminf_{n \to \infty} n {\hat{\lambda}_n}^{-1/q},
\]
 and from Lemma \ref{dn
lambdan} we have
\[ \limsup_{n \to \infty} n {\hat{\lambda}_n}^{-1/q} \le c_{pq} \left(\int_I |uv|^{1/r} dt
\right)^r  \] which finishes the proof.
\end{proof}

\section{The case $p\le q$}

\begin{lemma}
 \label{bn>lambdan}Let $1 < p \le q < \infty$ and $n >1$. Then
  $ b_n(T) \ge \check{\lambda}^{-1/q}, $ where
  $\check{\lambda} =
\min (sp_n(p,q)).$
\end{lemma}
\begin{proof} We use the construction of Buslaev \cite{B1}
Take $(\check{g},\check{f},\check{{\lambda}})$ from $SP_n(T,p,q)$
and denote by $a=x_0<x_1< ...<x_i< ... <x_n<x_{n+1}=b$ the zeros
of $\check{g}$. Set $I_i=(x_{i-1},x_i)$ for $1 \le i \le n+1$,
 $f_i(.)=\check{f}(.)\chi_{I_i}(.)$ and $g_i(.)=\check{g}(.)\chi_{I_i}(.)$.
 Then
$Tf_i=g_i(.)$ for $1\le i \le n+1$.

Define $X_{n+1}=\spann\{f_1, ... f_{n+1}\}$. Since the supports of
$\{f_i\}$ and $\{g_i\}$ are disjoint, then we have
\[b_n(T)\ge  \inf_{\alpha \in \mathbb{R}^n \setminus \{0\}}
{{\|T \left( \sum_{i=1}^{n+1} \alpha_i f_i \right) \|_{q,I}} \over
{\| \sum_{i=1}^{n+1} \alpha_i f_i \|_{p,I}}}=
 \inf_{\alpha \in
\mathbb{R}^n \setminus \{0\}} {{\| \sum_{i=1}^{n+1} \alpha_i g_i
\|_{q,I}} \over {\| \sum_{i=1}^{n+1} \alpha_i f_i \|_{p,I}}} .
\]
We shall study the extremal problem of finding
\[\inf_{\alpha \in
\mathbb{R}^n \setminus \{0\}} {{\| \sum_{i=1}^{n+1} \alpha_i g_i
\|_{q,I}} \over {\| \sum_{i=1}^{n+1} \alpha_i f_i \|_{p,I}}} .
\]
It is obvious that the extremal problem has a solution. Denote
that solution by $\bar{\alpha}=(\bar{\alpha}_1, \bar{\alpha}_2,
... )$. Since $p\le q$, a short computation shows us that
$\bar{\alpha}_i \not = 0$ for every $i$, moreover we can suppose
that  the $\bar{\alpha}_i$ alternate in sign. Label
\[\bar{\gamma}: =  {{\| \sum_{i=1}^{n+1} \bar{\alpha}_i g_i \|^q_{q,I}} \over {\|
\sum_{i=1}^{n+1} \bar{\alpha}_i f_i \|^p_{p,I}}};
\]
then the solution of the extremal problem is given by
$\bar{g}=\sum_{i=1}^{n+1} \bar{\alpha}_i g_i$,
$\bar{f}=\sum_{i=1}^{n+1} \bar{\alpha}_i f_i$ where
$\|\bar{f}\|_p=1$.

 Let us take the vector $\beta=(1,-1,...
)$. Define the functions $\tilde{g}=\sum_{i=1}^{n+1} \beta_i g_i$,
$\tilde{f}=\sum_{i=1}^{n+1} \beta_i f_i$. Then
\[\lambda_n^{-1}: =  {{\| \sum_{i=1}^{n+1} {\beta}_i g_i \|^q_{q,I}} \over {\|
\sum_{i=1}^{n+1} {\beta}_i f_i \|^p_{p,I}}}.
\]
 It is obvious that
$\bar{\gamma} \le \lambda_n^{-1}$. Suppose that $\bar{\gamma} <
\lambda^{-1}$.

Since $\bar{\alpha}_i \not = 0$, $|\beta_i|=1$ and $\bar{\gamma} <
\lambda^{-1}$ then $0<\varepsilon^*:=\min_{1\le i \le
n+1}(\beta_i/\bar{\alpha}_i)<1$. From Lemma \ref{corollary 3 BT}
follows
\[P(T (\tilde{f}) - \varepsilon^* T (\bar{f}))\le
P(T (\tilde{f}) - {\varepsilon^*}^{(p-1)/(q-1)} (\bar{\gamma} /
\lambda_n^{-1})^{1/(q-1)} T (\bar{f})).
\]
By repeated use of Lemma \ref{corollary 3 BT} with the help of
$(\varepsilon^*)^{(p-1)/(q-1)}\le \varepsilon^* <1$ and
$\bar{\gamma} / \lambda^{-1} <1$ we  get
\[P(T (\tilde{f}) - \varepsilon^* T (\bar{f}))\le P(T (\tilde{f}))
=n. \]

On the other hand we have from Lemma \ref{lemma 3.6 BT} and the
definition of $\varepsilon^* $ that
\[ P(T (\tilde{f}) - \varepsilon^* T (\bar{f}))\le
P(\tilde{f} - \varepsilon^* \bar{f}) = P(\sum_{i=1}^{n+1} \beta_i
f_i-
 \varepsilon^*
\sum_{i=1}^{n+1} \bar{\alpha}_i f_i ) \le n-1,
\] which contradicts $\bar{\gamma} < \lambda^{-1}$.
\end{proof}

\begin{theorem}
  \label{p<q} If $1< p\le q < \infty$ then
  \[ \lim_{n \to \infty} n {\check{\lambda}_n}^{-1/q}
  =c_{pq}\left(\int_I |uv|^{r} dt \right)^{1/r}\]
  where $r=1/p'+1/q$, ${\check{\lambda}_n}=
   \min (sp_n(p,q))$ and $c_{pq}$ as in (\ref{constant}).
\end{theorem}
\begin{proof}
From \cite{EL1} we have
\[\lim_{n \to \infty} n
{b_n(T)}=c_{pq}\left(\int_I |uv|^{r} dt \right)^{1/r} \] and since
 $b_n(T)\searrow 0$
then from Lemma \ref{bn<lambdan} it follows that
\[c_{pq}\left(\int_I |uv|^{r} dt
\right)^{1/r} \le \liminf_{n \to \infty} n
{\check{\lambda}_n}^{-1/q}. \]
 Moreover, from Lemma \ref{bn>lambdan} we have
\[ \limsup_{n \to \infty} n {\check{\lambda}_n}^{-1/q} \le c_{pq} \left(\int_I |uv|^{r} dt
\right)^{1/r}  \] which finishes the proof.
\end{proof}

When $p=q$ the following lemma follows from Theorem \ref{q<p} and
Theorem \ref{p<q} (we can find this result in a sharper form in
\cite{B}).

\begin{remark}
When $p=q$ then
 \[ \lim_{n \to \infty} n {\lambda_n}^{-1/q}
  =c_{pq}\left(\int_I |uv|^{r} dt \right)^{1/r}\]
  where $r=1/p'+1/q$, $c_{pq}$ as in (\ref{constant}) and ${\lambda}_n$ is the single point in
   $sp_n(p,q).$
\end{remark}

{\bf Address of authors:}

   D.E. Edmunds\\
School of Mathematics\\
Cardiff University\\
Senghennydd Road\\
   CARDIFF CF24 4YH \\
   UK\\
   e-mail: {\verb"DavidEEdmunds@aol.com"}

\

   J. Lang\\
   Department of Mathematics\\
    The Ohio State University\\
   100 Math Tower\\
    231 West 18th Avenue\\
   Columbus, OH 43210-1174\\
   USA \\
   e-mail: {\verb"lang@math.ohio-state.edu"}

\begin{thebibliography}{99}                                                                                               %

\bibitem{B} C. Bennewitz, Approximation numbers = Singular values,
Journal of Computational and Applied Mathematics, to appear.

\bibitem{B1}A.P. Buslaev, On Bernstein-Nikol'ski\u{\i} inequalities and
widths of Sobolev classes of functions, Dokl. Akad. Nauk
\textbf{323} (1992), no. 2, 202-205.

\bibitem{BT1}A.P. Buslaev and V.M. Tikhomirov, Spectra of nonlinear
differential equations and widths of Sobolev classes. Mat. Sb.
\textbf{181} (1990), 1587-1606; English transl. in Math. USSR Sb.
\textbf{71 }(1992), 427-446.


\bibitem{EE1}Edmunds, D.E. and Evans, W.D., \textit{Spectral theory and
differential operators, }Oxford University Press, Oxford, 1987.


\bibitem{EE2}D.E. Edmunds and W.D. Evans, \textit{Hardy operators, function
spaces and embeddings, }Springer, Berlin-Heidelberg-New York,
2004.

\bibitem{EL1} D.E. Edmunds and J. Lang, Bernstein widths of Hardy-type operators
in a non-homogeneous case, to be published in Journal of
Mathematical Analysis and Applications.

\bibitem{EL2} D.E. Edmunds and J. Lang, Approximation numbers and Kolmogorov
widths of Hardy-type operators in a non-homogeneous case,
Mathematische Nachrichten, \textbf{279} (2006), no. 7, 727-742.

\bibitem{N1}Nguen, T'en Nam, The spectrum of nonlinear integral equations and
widths of function classes, Math. Notes \textbf{53 }(1993), no.
3-4, 424-429.

\bibitem{KJF} A. Kufner, O. John and S. Fucik, \it{ Function
spaces}, Noordhoff International Publishing, Leyden, (1977).


\end{thebibliography}
\end{document}